\documentclass[a4paper,12pt]{article}
\usepackage[all]{xy} 
\usepackage{amssymb}
\usepackage{epsfig}
\usepackage{amsfonts}
\usepackage{amsmath}
\usepackage{euscript}
\usepackage{amscd}
\usepackage{amsthm}
\DeclareMathAlphabet{\mathpzc}{OT1}{pzc}{m}{it}

\newtheorem{thm}{Theorem}[section]
\newtheorem{lem}[thm]{Lemma}
 
\newtheorem{cor}[thm]{Corollary}
\newtheorem{rem}[thm]{Remark}

\newtheorem{defn}[thm]{Definition}

\newcommand{\bZ}{\mathbb Z}

\newcommand{\bC}{\mathbb C}
\newcommand{\bN}{\mathbb N}

\newcommand{\A}{\mathbb A}

\newcommand{\aut}{\operatorname{Aut_k}}

\newcommand{\td}{\operatorname{tr.deg}}

\newcommand{\gr}{\operatorname{gr}}

\newcommand{\ml}{\operatorname{ML}}

\title{On Double Danielewski Surfaces\\ and\\ the Cancellation Problem}
\author{Neena Gupta$^*$ and Sourav Sen$^\dagger$\\
{\small{\it $^*$Stat-Math Unit, Indian Statistical Institute,}}\\
{\small{\it 203 B.T. Road, Kolkata 700 108, India.}}\\
{\small{\it e-mail : neenag@isical.ac.in}}\\
{\small{\it $^\dagger$Swami Vivekananda Research Centre,}}\\
{\small {\it Ramakrishna Mission Vidyamandira }}\\
{\small {\it P.O. Belur Math, Howrah 711202, India.}}\\
{\small{\it e-mail : sourav.sen3@gmail.com}}
}

\begin{document}

\date{}
\maketitle
\abstract{
We study a two-dimensional family of affine surfaces which are counter-examples to the
Cancellation Problem. We describe the Makar-Limanov invariant of these surfaces, determine their isomorphism classes and characterize
the automorphisms of these surfaces.
}
\smallskip

\noindent
{\small {{\bf Keywords}. Cancellation Problem, Exponential Maps, Makar-Limanov Invariant, Automorphism, Stable Isomorphism.}
	
}

\section{Introduction}
For integral domains $C \subset A$, the notation
$A=C^{[n]}$ will mean that $A=C[t_1, \ldots, t_n]$ for elements
$t_1, \ldots , t_n \in A$ algebraically independent over $C.$

Let $k$ be a field.
A version of the Cancellation Problem asks:

\smallskip
\noindent
{\bf Question 1}: If $A$ and $B$ are two 
finitely generated $k$-algebras such that $A^{[1]}\cong_k B^{[1]}$, does
this necessarily imply that $A \cong_k B$?

\smallskip

For a finitely generated $k$-algebra $B$,  if there exists a $k$-algebra $A$ such that\\
$A^{[1]}\cong_k B^{[1]}$ but $A \ncong_k B$, then we say that $B$ 
does not satisfy the cancellation property. 

The Cancellation Problem is known to have an affirmative solution when $\dim B=1$ (see \cite{AEH}) 
but in higher dimensions there are known counter-examples to this problem.
In \cite{H}, M. Hochster showed that the coordinate ring of the tangent bundle over the real sphere does not satisfy the cancellation property. 
However,  T. Fujita, M. Miyanishi and T. Sugie (\cite{F}, \cite{MS})
proved the cancellation property of $k^{[2]}$ for any field $k$ of characteristic zero 
and P. Russell (\cite{R}) extended their results over perfect fields of arbitrary characteristic.
S.M. Bhatwadekar and the first author (\cite{BG}) established 
the cancellation property of $k^{[2]}$ over any arbitrary field $k$.
When ch. $k>0$ and $r \geq 3$, it has been shown that $B= k^{[r]}$ does not 
satisfy the cancellation property (see \cite{G}, \cite{G2}).
When ch. $k=0$ and $r \geq 3$, it is not known whether the polynomial ring $B= k^{[r]}$
has the cancellation property. 

In 1989, W. Danielewski constructed explicit examples (\cite{Da}) of two-dimensional affine domains
over the field of complex numbers $\bC$ which do not satisfy the cancellation property. For any non-constant 
polynomial $P(Z)$ with distinct roots, Danielewski considered the coordinate ring of the affine surface 
$S_n$ defined by the equation $x^ny-P(z)=0$ in $\A^3_k$. Such rings are known as Danielewski  surfaces. 
It is known that for any pair $(m,n)$ with $m \ne n$, $S_m \ncong S_n$ but $S_m \times \A^1_k \cong S_n \times \A^1_k$ 
(cf. \cite{Fi}, \cite[p. 246]{Fr}).


In this paper, we study a family of two-dimensional affine surfaces $S_{d, e}$ over a field 
$k$ (of any characteristic), defined by a pair of equations 
$$
\{x^{d}y - P(x,z)=0,x^{e}t - Q(x,y,z)=0\} ~\text{in}~ {\A}^4_k,
$$
where $d,e \in{ \bN}$, $P(X,Z)$ is monic in $Z$ and $Q(X,Y,Z)$ is monic in $Y$
with $\deg_ZP(X,Z)\geq 2, \deg_YQ(X,Y,Z)\geq 2$. 
We call them ``double Danielewski surfaces''.
We first compute the $\ml$-invariant of double Danielewski surfaces (Theorem \ref{ml}). 
Next, we  determine the isomorphism classes of these surfaces explicitly  (Theorem \ref{isomclass})
and describe a characterization of their automorphisms (Theorems \ref{auto} and \ref{autoauto}).
We also deduce that no double Danielewski  surface is isomorphic to any Danielewski  surface  (Corollary \ref{cdistinct}).
Finally, we prove a stable isomorphism property of  the coordinate rings of double Danielewski surfaces 
under certain regularity assumptions (Theorem \ref{stable}) and hence deduce that such rings do not satisfy the cancellation property
 (Corollary \ref{cpd}).

\section{Exponential maps}
In this section we recall the basics of exponential maps.
 
\smallskip
\noindent
{\bf Definition.}
Let $k$ be a field, $A$ be a $k$-algebra and 
let $\phi: A \to A^{[1]}$ be a 
$k$-algebra homomorphism. For an indeterminate $U$ over $A$, let the notation 
$\phi_U$ denote the map $\phi: A \to A[U]$.  
$\phi$ is said to be an {\it exponential map on $A$} if $\phi$ satisfies the following two properties:
\begin{enumerate}
 \item [\rm (i)] $\varepsilon_0 \phi_U$ is identity on $A$, where 
 $\varepsilon_0: A[U] \to A$ is the evaluation at $U = 0$.
 \item[\rm (ii)] $\phi_V \phi_U = \phi_{V+U}$, where 
 $\phi_V: A \to A[V]$ is extended to a homomorphism 
 $\phi_V: A[U] \to A[V,U]$ by  setting $\phi_V(U)= U$.
 \end{enumerate}
The subring 
$A^{\phi}= \{a \in A\,| \,\phi (a) = a\}$ of $A$ is said to be the ring of invariants of $\phi$. 

An exponential map $\phi$ is said to be {\it non-trivial} if $A^{\phi} \neq A$.
For an affine domain $A$ over a field $k$, let 
${\rm EXP} (A)$ denote the set of all exponential maps on $A$.
The  {\it Makar-Limanov invariant} of 
$A$ is a subring of $A$ defined by
$$
{\rm ML} (A) = \bigcap_{\phi \in {\rm EXP} (A)} A^{\phi}.
$$ 

We summarise below some useful properties of an exponential map $\phi$ 
(cf. \cite[p. 1291--1292]{C} and \cite[Lemma 2.8]{G2}).

\begin{lem}\label{exp3}
Let $A$ be an affine domain over a field $k$. Suppose that there exists a non-trivial
exponential map $\phi$ on $A$. Then the following statements hold:
 \begin{enumerate}
\item [\rm (i)] $A^{\phi}$ is factorially closed in $A$.
\item [\rm (ii)] $A^{\phi}$ is algebraically closed in $A$.
\item [\rm (iii)] If $x \in A$ is such that $\deg_U\phi(x)$ is of minimal positive degree,
and $c$ is the leading coefficient of $U$ in $\phi(x)$, then $c \in A^{\phi}$ and 
$A[c^{-1}]= A^{\phi}[c^{-1}][x]$.
\item[\rm (iv)] $\td_k (A^{\phi}) = \td_k (A)-1$. 
\item [\rm (v)] If $\td_k(A)=1$ and $\tilde{k}$ is the algebraic closure of $k$ in $A$,
then  $A= \tilde{k}^{[1]}$ and $A^{\phi} = \tilde{k}$. 
\item[\rm (vi)] For any multiplicative subset $S$ of $A^{\phi}\setminus \{0\}$, 
$\phi$ extends to a non-trivial exponential map $S^{-1}\phi$ on $S^{-1}A$ by setting 
$(S^{-1}\phi) (a/s) := \phi(a)/s$ for $a \in A$, $s \in S$; 
and the ring of invariants of $S^{-1}\phi$ is $S^{-1}(A^{\phi})$.
\end{enumerate}
\end{lem}

We recall below the concept of an admissible proper $\bZ$-filtration 
on an affine domain (cf. \cite{C}). 

\smallskip
\noindent
{\bf Definition.}
Let $A$ be an affine domain over a field $k$.
A collection of $k$-linear subspaces $\{A_n\}_{n \in \bZ}$ of $A$ 
is said to be a {\it proper $\bZ$-filtration} if it satisfies the following conditions: 
\begin{enumerate}
\item [\rm (i)] $A_n \subseteq A_{n+1}$ for all $n \in \bZ$,
\item [\rm (ii)]  $A= \bigcup_{n \in \bZ} A_n$,
\item [\rm (iii)] $\bigcap_{n\in \bZ} A_n = (0)$ and 
\item [\rm (iv)]  $(A_n\setminus A_{n-1})\cdot (A_m \setminus A_{m-1}) 
\subseteq A_{n+m} \setminus A_{n+m-1}$ for all $n, m \in \bZ$.
\end{enumerate}

We shall call a proper $\bZ$-filtration $\{A_n\}_{n \in \bZ}$ of $A$ 
{\it admissible} if there exists a finite generating set $\Gamma$
of $A$ such that, for any $n \in \bZ$ and $a \in A_n$, $a$ can be written
as a finite sum of monomials in elements of $\Gamma$ and each of these monomials 
is an element of $A_n$. 

Any proper $\bZ$-filtration on $A$ determines the following $\bZ$-graded integral domain 
$$
\gr (A) : = \bigoplus_i A_i/A_{i-1},
$$
and a map
$$
\rho: A \to \gr (A)~~ {\text{defined by}}~~ \rho(a) = a + A_{n-1}, ~~{\rm if}~~ a \in A_n \setminus A_{n-1}.
$$

An exponential map $\phi$ on a graded ring $A$ is said to be {\it homogeneous}
if $\phi: A \to A[U]$ becomes homogeneous 
when $A[U]$ is given a grading induced from $A$ such that $U$ is a homogeneous element. 

\begin{rem}\label{homin}
{\em
Note that if $\phi$ is a homogeneous exponential map on a graded ring $A$, 
then $A^{\phi}$ is a graded subring of $A$. 
}
\end{rem}

\begin{defn}
{\em Let $A$ be an affine domain over a field $k$. A \textit{$\bZ$-grading} of $A$
is a family $\{A_n\}_{n\in \bZ}$ of subgroups of $(A,+)$ such that: 
\begin{enumerate}
\item[\rm(i)] $A= \bigoplus_{n\in \bZ}A_n$.
\item[\rm(ii)] $A_nA_m \subseteq A_{n+m}$ for all $n,m\in \bZ$.
\end{enumerate}}
\end{defn}

Finally, we state below a result on homogenization of exponential maps due to 
H. Derksen, O. Hadas and L. Makar-Limanov (\cite{DHM}) (cf. \cite[Theorem 2.3]{G}).

\begin{thm}\label{MDH}
Let $A$ be an affine domain over a field $k$ with an admissible 
proper $\bZ$-filtration and $\gr(A)$ the induced $\bZ$-graded domain.
Let $\phi$ be a non-trivial exponential map on $A$.
Then $\phi$ induces a non-trivial homogeneous exponential map $\bar{\phi}$ on 
$\gr (A)$ such that $\rho( {A^{\phi}}) \subseteq {\gr (A)}^{\bar{\phi}}$.
\end{thm}

\section{Main Theorems}
Throughout the section $k$ 
will denote a field and $B$ will denote the ring
\begin{equation}\label{B}
B= \dfrac{k[X,Y,Z,T]}{(X^{d}Y - P(X,Z), X^{e}T - Q(X,Y,Z))}, 
\end{equation}
where $d,e \in \bN$, $P(X,Z)\in k[X,Z]$ is monic in $Z$ and $Q(X,Y,Z)\in k[X,Y,Z]$ is monic in $Y$.
The letters $x,y,z$ and $t$ will denote the images of $X,Y,Z$ and $T$ respectively in $B.$ 
Set $r := \deg_Z P(X,Z)$ and $s := \deg_Y Q(X,Y,Z).$ Note that when $r=1$ or $s=1$,  
the ring $B$ is a Danielewski surface. We call $B$ a ``double Danielewski surface" if  $r \geq 2$ and $s \ge 2$.
For a ring $R$, the notation $R^*$ will denote the group of units of $R$.

We will compute $\ml(B)$ in Section \ref{mlb} and discuss the isomorphism classes of double Daneilewski surfaces
and characterize the automorphisms of $B$ in Section \ref{ISC}.

\subsection{$\ml$-invariant of $B$}\label{mlb}

For convenience, we state below an elementary result.
\begin{lem}\label{rs}
 Let $R$ be an integral domain and $a, b \in R\setminus \{0\}$. If $a$ is  not a zero-divisor on $R/(b)$,
then $b^n$ is  not a zero-divisor on $R/(a)$ for any integer $n\ge 1$. 
\end{lem}
\begin{proof}
If $b^n \alpha= a \beta$ for some $\alpha, \beta \in R$, then as $a$ is  not a zero-divisor on $R/(b)$,
we have $\beta \in bR$, and hence since $R$ is an integral domain, we have $b^{n-1}\alpha \in aR$. 
Proceeding in a similar manner, we will get that $\alpha \in aR$. Hence $b^n$ is  not a zero-divisor on $R/(a)$.
\end{proof}

We now recall another elementary lemma ({\cite[Lemma 2.4(2)]{DGO}}).
\begin{lem}\label{lemm}
	Let $R$ be an integral domain and $a, b \in R\setminus \{0\}$. If $b$ is not a zero-divisor on $R/(a)$,
then the ring $\dfrac{R[T]}{(bT-a)}$ is an integral domain.
\end{lem}

\begin{lem}{\label{ID}}
$B$ is an integral domain.
\end{lem}

\begin{proof}
Let $R= \dfrac{k[X,Y,Z]}{(X^{d}Y-P(X,Z))}.$ Since $(X^{d}Y-P(X,Z))$ is linear in
 $Y$ and $X$ does not divide $P(X,Z)$ in $k[X,Z]$, $(X^{d}Y-P(X,Z))$ is irreducible
  in the UFD $k[X,Y,Z]$. Hence, $R$ is an integral domain. We can identify $R$ as a
   subring of $B$, by identifying the images of $X,Y,Z$ in $R$ with $x, y, z$ in $B$. Then $B=R[T]/(x^{e}T-Q(x,y,z))$. 

Now $R/(x) \cong \left(\dfrac{k[Z]}{P(0,Z)}\right)[Y]$. Therefore, since
$Q(X,Y,Z)$ is monic in $Y$, it follows that $Q(x,y,z)$  is  not a zero-divisor on $R/(x)$.
Hence by Lemma \ref{rs}, $x^e$ is  not
a zero-divisor on $R/(Q(x,y,z))$. 
 Therefore, by Lemma \ref{lemm}, $B$ is an integral domain.
\end{proof}

In the next two results we show that there exists an admissible proper $\bZ$-filtration on $B$ such that $gr(B)$ is isomorphic to
a special case of the ring $B$ which we denote by $D$ and an 
admissible proper $\bZ$-filtration on $D$ such that $gr(D)$ is isomorphic to
a further special case of the ring $D$ which we denote by $C$. 
Note that in these results $P(X,Z)$ is as in $B$. 
Also recall that $s$ denotes $\deg_YQ(X,Y,Z)$ and $r$ denotes $\deg_ZP(X,Z)$.

\begin{lem}\label{filt1}
Considering $B$ as a subring of the $\bZ$-graded ring $k[x,{x}^{-1},z] = \bigoplus_{i \in \bZ}k[z]x^{i}$,
define a proper $\bZ$-filtration $\{B_n\}_{n \in \bZ}$ on $B$ by 
$$
B_{n} := B \cap (\bigoplus_{i\geq -n}k[z]x^{i}).
$$
This filtration on $B$ is admissible with the generating set $\{x,y,z,t\}$
and the corresponding graded ring $\gr(B)  = \bigoplus_{n \in \bZ} (B_{n}/B_{n+1})$
is isomorphic to $$D:= \dfrac{k[X,Y,Z,T]}{(X^{d}Y - P(0,Z), X^{e}T - Y^s)}.$$
\end{lem}

\begin{proof}
We have  $x \in B_{-1}\setminus B_{-2}$, 
$z \in B_{0}\setminus B_{-1}$, $y \in B_{d}\setminus B_{d-1}$ and 
$t \in B_{(ds+e)} \setminus B_{(ds+ e-1)}.$ 
Using the relations $x^{d}y=P(x,z)$ and $x^{e}t=Q(x,y,z),$ we see that each element $g \in B$ can be written as 
\begin{equation}\label{e1}
g= f_0(x,z) + \sum_{\substack{0 \leq i <d\\ j>0}} a_{ij}(z)x^{i}y^{j} + \sum_{\substack{0\leq i <e\\l>0}} b_{il}(z)x^{i}t^{l} + 
\sum_{\substack{0\leq i< min\{d,e\}\\j,l>0}} c_{ijl}(z)x^{i}y^{j}t^{l} ,
\end{equation}
where $f_{0}(x,z) \in k[x,z]$, $a_{ij}(z) , b_{il}(z), c_{ijl}(z) \in k[z].$
Let $\widetilde{B}$ denote the graded ring $\gr(B) = \bigoplus_{n \in \bZ} (B_{n}/B_{n+1})$ 
with respect to the above filtration.
For $g \in B,$ let $\tilde{g}$ denote the image of $g$ in $\widetilde{B}.$ 
It follows from  (\ref{e1}), 
that the filtration defined on $B$ is admissible 
with the generating set $\Gamma = \{x,y,z,t\}.$ Hence, 
$\widetilde{B}$ is generated by $\tilde{x}, \tilde{y}, \tilde{z}$ and $\tilde{t}.$ 

We now show that, $\widetilde{B} \cong D$.
Note that, $x^{d}y, P(0,z) \in B_{0}.$ Hence, since  $x^{d}y- P(0,z) \in B_{-1}$, 
we have $\tilde{x}^{d}\tilde{y} - P(0,\tilde{z})=0$ in $\widetilde{B}$. 
Again note that  $x^{e}t, y^{s} \in B_{ds}$ and $x^{e}t- y^{s} = Q(x, y, z) - y^s \in   B_{ds-1}$. 
Hence,  $\tilde{x}^{e}\tilde{t}- \tilde{y}^{s}=0$ in $\widetilde{B}.$ 
As $\widetilde{B}$ can be identified with a subring of $\gr(k[x,{x}^{-1},z]) \cong k[x,{x}^{-1},z],$ 
we see that the elements $\tilde{x}$ and $\tilde{z}$ of $\widetilde{B}$ are algebraically independent over $k.$
Since $D$
is an integral domain (cf. Lemma \ref{ID}),
we have, $\widetilde{B} \cong D$.
\end{proof}

\begin{lem}\label{filt2}
Let $D$ be as in Lemma \ref{filt1} and let $\tilde{x}, \tilde{y}, \tilde{z}, \tilde{t}$
respectively denote the images of $X, Y,Z, T$ in $D$. 
Considering $D$ as a subring of the $\bN$-graded ring 
$k[\tilde{x},\tilde{x}^{-1},\tilde{z}] = \bigoplus_{i \in \bN}k[\tilde{x}, {\tilde{x}}^{-1}]\tilde{z}^i,$
define a  proper $\bZ$-filtration $\{D_n\}_{n \in \bZ}$ on
 $D$ by 
$$
D_{n} := D\cap (\bigoplus_{i\le n}k[\tilde{x}, {\tilde{x}}^{-1}]\tilde{z}^i).
$$
This filtration on $D$ is admissible with the generating set $\{\tilde{x}, \tilde{y}, \tilde{z}, \tilde{t}\}$
and the corresponding graded ring $\gr(D)  = \bigoplus_{n \in \bZ} (D_{n}/D_{n+1})$
is isomorphic to $$C:= \dfrac{k[X,Y,Z,T]}{(X^{d}Y - Z^r, X^{e}T - Y^s)}.$$
\end{lem}
\begin{proof} 
Note that, $D_{n} =\phi$ for all $n<0$,  
$\tilde{z}\in D_1\setminus D_0$,  $\tilde{x} \in D_0 \setminus D_{-1}$, 
$\tilde{y} \in D_{r} \setminus D_{r-1}$ and $\tilde{t} \in D_{rs}\setminus D_{rs-1}.$
 Using the relations $\tilde{x}^{d}\tilde{y} = P(0,\tilde{z})$ and $\tilde{x}^{e}\tilde{t}= \tilde{y}^{s},$ 
 we see that each element $\tilde{g} \in D$ can be written  as 
\begin{equation}\label{e2}
\tilde{g} = \sum_{i=0}^{r-1}\left(\sum_{0\le j<s}g_{ij}(\tilde{x}){\tilde{y}}^j+
\sum_{\substack{0\le j< s \\\ell> 0}}h_{ij\ell}(\tilde{x}){\tilde{y}^j}\tilde{t}^{\ell}\right)\tilde{z}^i,
\end{equation}
where $g_{ij}(\tilde{x}), h_{ij\ell}(\tilde{x}) \in k[\tilde{x}]$.
Let $\bar{D}$ denote the graded ring $\gr(D)= \bigoplus_{n \in \bZ}(D_n/D_{n-1})$ 
with respect to the above filtration. For $\tilde{g} \in D,$ let $\bar{g}$ denote the image of $\tilde{g}$ in $\bar{D}.$
It follows from (\ref{e2}) that the filtration defined on $D$ is admissible 
with the generating set $\Gamma^{'} = \{\tilde{x},\tilde{y},\tilde{z},\tilde{t}\}.$ 
Hence, $\bar{D}$ is generated by $\bar{x}$, $\bar{y}$, $\bar{z}$ and $\bar{t}.$ 

We now show that $\bar{D} \cong C$. 
Note that $\tilde{x}^{d}\tilde{y}$, $\tilde{z}^{r} \in D_{r}$ and
 $\tilde{x}^{d}\tilde{y}-\tilde{z}^{r} = P(0,\tilde{z})-\tilde{z}^r \in D_{r-1}$.
Hence, $\bar{x}^{d}\bar{y} - \bar{z}^{r}=0$ in $\bar{D}$.
 Again, $\tilde{x}^{e}\tilde{t}$, $\tilde{y}^{s} \in D_{rs}$ and $\tilde{x}^{e}\tilde{t}-\tilde{y}^{s}=0$ in 
$D$.
 Hence, $\bar{x}^{e}\bar{t} - \bar{y}^{s}=0$ in $\bar{D}.$ As $\bar{D}$ can be identified with a subring of 
 $\gr(k[\tilde{x}, {\tilde{x}}^{-1},\tilde{z}]) \cong k[\tilde{x}, {\tilde{x}}^{-1},\tilde{z}],$ 
 we see that the elements $\bar{x}$ and $\bar{z}$ of $\bar{D}$ are algebraically independent over $k.$ 
 Since $C$ is an integral domain (cf. Lemma \ref{ID}),
  we have, $\gr(D)= \bar{D} \cong C.$
\end{proof}

\begin{lem}\label{mll}
Let $C$ be the integral domain defined by 
$$
\dfrac{k[X,Y,Z,T]}{(X^{d}Y-Z^{r},X^{e}T-Y^{s})}, {\text{~~where~~}} d, e \ge 1;
$$
and any one of the following holds:
\begin{eqnarray}\label{mlc}
\text{~either~} r\ge 2 \text{~and~}  \nonumber s\ge 2 \\
\text{~or~} r \ge 2\text{~and~} s=1  \\
\text{~or~}r=1, s \ge 2 \text{~and~} e\ge 2 \nonumber .
\end{eqnarray}
Let $\bar{x}$, $\bar{y}$, $\bar{z}$ and $\bar{t}$ respectively denote the images of 
$X$, $Y$, $Z$ and $T$ in $C$. 
Consider $C = \bigoplus_{i \in \bZ} C_i$ as a graded subring of $k[\bar{x},\bar{x}^{-1}, \bar{z}]$  with 
$$C_i= C \cap k[\bar{x},\bar{x}^{-1}]\bar{z}^i {\text{~~for each~~}} i \geq 0 {\text{~~and~~}} C_i=0 {\text{~~for~~}} i< 0.$$ Then $C^{\phi}\subseteq k[\bar{x}]$
for any non-trivial homogeneous exponential map
$\phi$ on the graded ring $C$.
\end{lem}

\begin{proof}
Let $\phi$ be a $\bZ$-graded exponential map on $C$.
We note that this grading induces a degree function on $C$, with $\deg \bar{x}=0$, $\deg \bar{z}=1$, $\deg \bar{y}=r$ and $\deg \bar{t}=rs$.
Let 
$$
R= \dfrac{k[X,Y,T]}{(X^{e}T-Y^{s})}.
$$
We identify $R$ as a subring of $C$ identifying the images of 
$X$, $Y$ and  $T$ in $R$ with $\bar{x}$, $\bar{y}$ and $\bar{t}$ in $C$. Note that $R \hookrightarrow \bigoplus_{i \in r\bZ}C_i$.
We show that $C^{\phi} \subseteq R.$ 

We first note that any element $f \in C$ can be uniquely written as 
$$
f= \sum_{i=0}^{r-1}f_i\bar{z}^{i},
$$
for some $f_i \in R.$ 
For if, 
$$
\deg(f_i\bar{z}^{i})=\deg(f_j\bar{z}^{j}) \text{~for some~} 0\le i, j \le r-1,
$$
then $\deg(f_i) + i = \deg(f_j) + j.$  
Since $f_i,f_j \in R$,  we have 
$$
\deg(f_i)-\deg(f_j) \equiv 0 \mod r,
$$ i.e., 
$i-j \equiv 0 \mod r$
which implies $i=j$. 
 
Suppose, if possible, that $C^{\phi} \nsubseteq R$. 
Then, as $C^{\phi}$ is a graded subring of $C$, $f_i\bar{z}^{i} \in C^{\phi}$ 
for some $f_i \in R$ and $i >0.$ 
By Lemma \ref{exp3}(i),  $\bar{z} \in C^{\phi}.$ Using the relations $\bar{x}^{d}\bar{y}=\bar{z}^r$ 
and $\bar{x}^{e}\bar{t}=\bar{y}^s,$ we see that $\bar{x},\bar{y},\bar{t} \in C^{\phi}$, i.e., $\phi$ is trivial, which is a contradiction. 
Hence, $C^{\phi} \subseteq R.$

We now show that $C^{\phi} \subseteq k[\bar{x}]$.
Any element $g\in R$ can be written as 
$$
g= \sum_{0 \leq i < s} g_i(\bar{x})\bar{y}^i + \sum_{\substack{i > 0\\0 \leq j <s}}g_{ij}(\bar{x})\bar{t}^i\bar{y}^j,
$$
for some $g_i, g_{ij} \in k[\bar{x}].$ 
Note that 
$$
\deg(g_i(\bar{x})\bar{y}^i) = ir < sr  \text{~if~}  i<s \text{~and~}
\deg(g_{ij}(\bar{x})\bar{t}^i\bar{y}^j) = (irs + jr) \text{~if~} i >0,0 \leq j<s.
$$ 
Thus a homogeneous element of $C$ in $R$ is of the form 
$g_{i}(\bar{x})\bar{y}^i$ for some $0\le i< s$ or $g_{ij}(\bar{x})\bar{t}^i\bar{y}^j$ for some $i >0$ and $0\le j<r$.
As $C^{\phi}$ is a graded subring of $C$, we have either $g_{i}(\bar{x})\bar{y}^i \in C^{\phi}$ for some $0\le i< s$ 
or $g_{ij}(\bar{x})\bar{t}^i\bar{y}^j$ for some $i >0$ and $0\le j<r$.
Suppose $g_{ij}(\bar{x})\bar{t}^{i}\bar{y}^{j} \in C^{\phi}$ for some $i>0$ 
and $0 \leq j <s.$ 
Then $C^{\phi}$ being factorially closed in $C,$ 
$\bar{t}\in C^{\phi}$ and so $\phi$ extends to a non-trivial exponential map of the ring 
$A:=\dfrac{k(T)[X,Y,Z]}{(X^{d}Y-Z^{r},X^{e}T-Y^{s})}$  (cf. Lemma \ref{exp3}). 
But since one of the conditions  of (\ref{mlc}) is satisfied, the ring $A$ 
is a non-normal ring of dimension one. Hence $\phi$ must be a trivial map (cf. Lemma \ref{exp3}(iii)), 
which is a contradiction. Hence, $\bar{t} \notin C^{\phi}.$
	
Therefore, $g_{i}(\bar{x})\bar{y}^{i} \in C^{\phi}$ for some $0 \leq i <s.$ If $i>0,$ then $C^{\phi}$ being factorially closed in $C,$ 
we have, $\bar{y} \in C^{\phi}$. Using the relations $\bar{x}^{e}\bar{t}=\bar{y}^{s}$ and 
$\bar{x}^{d}\bar{y}=\bar{z}^{r},$ we get, $\bar{x},\bar{z},\bar{t} \in C^{\phi},$ i.e., $\phi$ is trivial, which is a contradiction. 
Hence $i=0$ and $C^{\phi} \subseteq k[\bar{x}]$.
\end{proof}

\begin{lem}\label{exp1}
 There exists a non-trivial exponential map $\phi$ on $B$ such that $B^{\phi}=k[x]$. 
\end{lem}
\begin{proof}
Consider the  map $\phi: B \to B[U]$ defined  by, 
\begin{eqnarray*}\label{exponential}
\phi(x)&=&x,\\
\phi(z)&=& z+ x^{d+e}U,\\
\phi(y)&=& \dfrac{P(x,z+x^{d+e}U)}{x^{d}}= y + U\alpha(x,z,U),\\
\phi(t)&=&\dfrac{Q(x,y+U\alpha(x,z,U),z+x^{d+e}U)}{x^{e}}= t+ U\beta(x,y,z,U),
\end{eqnarray*} where $\alpha(x,z,U) \in k[x,z,U], \beta(x,y,z,U) \in k[x,y,z,U].$
It is easy to see that $\phi$ is an exponential map on $B$. 
Clearly $k[x]\subseteq B^{\phi}$. 
Since $k[x]$ is algebraically closed in $B$ and $\td_{k[x]}B=1,$ 
using Lemma \ref{exp3}, we see that $B^{\phi}=k[x]$. 
\end{proof}

\begin{thm}\label{ml} Let $B$ be as in (\ref{B}) and let the 
parameters $r,s$ and $e$ in $B$ satisfy the conditions (\ref{mlc}) of Lemma \ref{mll}. Then 
$\ml(B)=k[x]$.
\end{thm}
\begin{proof}
We show that if $\phi$ is any non-trivial exponential map of $B$, then $B^{\phi}=k[x]$.
 
Let $\phi$ be a non-trivial exponential map of $B$.  
Consider the admissible proper $\bZ$-filtration $\{B_n\}_{n \in \bZ}$ on $B$ defined in Lemma \ref{filt1}
and let $\rho$ denote the canonical map $B \to gr(B)=D$.
By Theorem \ref{MDH}, 
$\phi$ induces a non-trivial exponential map $\tilde{\phi}$
on $D$ such that $\rho(B^{\phi}) \subseteq {D}^{\widetilde{\phi}}$.
Let $f\in B^{\phi}$. Replacing $f$ by $f-\lambda$ for some 
$\lambda \in k^*$, we may assume that $\rho(f) \notin k$. 

Again consider the admissible proper $\bZ$-filtration 
$\{D_n\}_{n \in \bZ}$ of $D$
defined in Lemma \ref{filt2} and let $\bar{\rho}$ denote the canonical map $D \to \gr(D)= C$. 
By Theorem \ref{MDH}, $\tilde{\phi}$ induces a non-trivial exponential map $\bar{\phi}$
on $C$ such that $\bar{\rho}(D^{\tilde{\phi}}) \subseteq {C}^{\bar{\phi}}$. 
Therefore, $\bar{\rho}(\rho(f))\in  {C}^{\bar{\phi}}$. 
By Lemma \ref{mll},  $ {C}^{\bar{\phi}}\subseteq k[\bar{x}]$, where $\bar{x}$ denotes the image
of $x$ in $C$. Hence $\bar{\rho}(\rho(f)) \in k[\bar{x}] \subseteq C$.
It follows from the filtration defined on $D$ in Lemma \ref{filt2} and equation (\ref{e2}) 
that $\rho(f)\in k[\tilde{x}] \subseteq D$, where $\tilde{x}$ denotes the image of $x$ in $D$.
Again from the filtration defined on $B$ in Lemma \ref{filt1} and equation (\ref{e1}), it follows  
that $f \in k[x] \subseteq B$. Thus $B^{\phi} \subseteq k[x]$. 
Since $B^{\phi}$ is a factorially closed subring of $B$ of transcendence degree $1$ over $k$, we have,
$B^{\phi}=k[x]$.
This being true for any non-trivial exponential map $\phi$ on $B$, we have by Lemma \ref{exp1}, 
$\ml(B)=k[x]$.
\end{proof}

\begin{rem}
{\em 
Let $B$ be as in (\ref{B}) and suppose the parameters $r,s$ and $e$ in $B$ do not 
satisfy the conditions (\ref{mlc}) of Lemma \ref{mll}, i.e., either \{$r=s=1$\} or \{$r=e=1$ and $s \ge 2$\}.
If $r=s=1$, then $B \cong k^{[2]}$ and hence $\ml(B)=k$. If $r=e=1$ and $s \ge 2$, then 
$B \cong k[X,Y,Z]/(XZ-f(Y))$ for some polynomial $f(Y) \in k[Y]$. In this case also $\ml(B)=k$ (cf. \cite[p.247]{Fr}).
}
\end{rem}

\subsection{Isomorphism Classes}\label{ISC}
We now investigate isomorphism classes of a family of surfaces  which includes the double Danielewski surfaces.
We consider two such surfaces which, for convenience, we denote by $B_1$ and $B_2$
(not to be confused with the graded components of $B$ in Section \ref{mlb}):
$$
B_1= \dfrac{k[X,Y,Z,T]}{(X^{d_1}Y - P_1(X,Z), X^{e_1}T - Q_1(X,Y,Z))} 
$$
and 
$$
B_2= \dfrac{k[X,Y,Z,T]}{(X^{d_2}Y - P_2(X,Z), X^{e_2}T - Q_2(X,Y,Z))},
$$
where $d_1,e_1,d_2,e_2 \in \bN$, $P_1(X,Z),P_2(X,Z) \in k[X,Z]$ are monic polynomials in $Z$,
$Q_1(X,Y,Z)$, $Q_2(X,Y,Z) \in k[X,Y,Z]$ are monic polynomials in $Y$, 
with $r_i=\deg_ZP_i(X,Z)$ and  $s_i=\deg_YQ_i(X,Y,Z)$ for $i=1, 2$.
Let $x_1,y_1,z_1,t_1$ and $x_2,y_2,z_2,t_2$ denote the images of $X,Y,Z,T$ in $B_1$ and $B_2$ respectively.
Suppose that the conditions (\ref{mlc}) of Lemma \ref{mll} are satisfied by ($r_i, s_i, e_i$) for $i=1,2$. 
Then from Theorem \ref{ml}, $\ml(B_i)= k[x_i]$ for $i=1,2$. 

\begin{thm}\label{isomclass}
Suppose $B_1\cong B_2$. Then the following conditions hold:
\begin{enumerate}
\item[\rm(I)] $(d_1,e_1, r_1, s_1)=(d_2,e_2,r_2, s_2)$.  Let $(d,e,r,s)=(d_i, e_i, r_i, s_i)$ for $i=1,2$.
\item[\rm(II)] There exist $\lambda, \gamma \in k^*$, $\delta(X) \in k[X]$, $f(X,Z) \in k[X,Z]$ and $h(X,Y,Z) \in k[X,Y,Z]$ such that 
\begin{enumerate}
\item [\rm(i)] $P_2(\lambda X, \gamma Z+ \delta(X))= \tau P_1(X, Z)+ X^df(X,Z)$, where $\tau= \gamma^r (\in k^*)$.
In particular, $P_2(0,\gamma Z+ \delta(0)) = \tau P_1(0,Z)$. 
\item [\rm(ii)] $Q_2(\lambda X,\nu Y+g(X,Z),\gamma Z+ \delta(X))= \kappa Q_1(X,Y,Z)+ X^eh(X,Y,Z)$, 
where $\nu= \lambda^{-d}\tau$, $\kappa= \nu^s$ and $g(X,Z)= \lambda^{-d}f(X,Z)$. 
In particular,\\ $Q_2(0,\nu Y + g(0,Z), \gamma Z+ \delta(0)) = \kappa Q_1(0,Y,Z)$.
\end{enumerate}
\end{enumerate}
Moreover, if $\psi: B_2\to B_1$ is an isomorphism, then 
$$
\psi(x_2)=\lambda x_1,~ \psi(z_2)=\gamma z_1+ \delta(x_1),
$$
$$
 ~\psi(y_2)= \nu y_1+g(x_1, z_1)~\text{and}~\psi(t_2)= \lambda^{-e}(\kappa t_1+h(x,y,z)).
$$
Conversely, if conditions (I) and (II) hold, then $B_1 \cong B_2$. 
\end{thm}

\begin{proof}
Let $\psi : B_1 \rightarrow B_2$ be a $k$-algebra isomorphism. Replacing $B_1$ by $\psi(B_1)$, 
we may assume that $B_1=B_2=B$. By Theorem \ref{ml}, $\ml(B) = k[x_1] = k[x_2]$ and hence 
$$
x_2= \lambda x_1 + \mu
$$
for some $\lambda \in k^{*}, \mu \in k$ and $k(x_1)[z_1]=k(x_2)[z_2]$. 
Thus, since $B \cap k(x_1)=k[x_1]$, we have, $z_2 = \gamma z_1 + \delta$ for some $\gamma(x_1), \delta(x_1) \in k[x_1]$. Using symmetry, 
we have, $\gamma(x_1) \in k^*$, i.e., 
\begin{equation}\label{z2}
z_2 = \gamma z_1 + \delta(x_1)
\end{equation}
for some $\gamma \in k^{*}, \delta(x_1) \in k[x_1]$. Hence, 
\begin{equation}\label{kxz}
k[x_1,z_1] = k[x_2,z_2].
\end{equation}
As $y_2  \in B \subseteq k[x_1,{x_1}^{-1},z_1]$, 
there exists an integer $n \geq 0$ such that $x_1^{n}y_2 \in k[x_1,z_1]$.
Therefore, since  
$$
{P_2(x_2,z_2)}= {x_2^{d_2}}y_2= (\lambda x_1 + \mu)^{d_2}y_2,
$$
 we have
${x_1^{n}P_2(x_2,z_2)} \in {(\lambda x_1 + \mu)^{d_2}}k[x_1,z_1]$. 
If $\mu \neq 0$, then $(\lambda x_1 + \mu ) \arrowvert P_2(x_2,z_2)$ in $k[x_1,z_1]$, i.e., 
$(\lambda x_1 + \mu ) \arrowvert P_2(\lambda x_1 + \mu,\gamma z_1 + \delta(x_1) )$ in 
$k[x_1,z_1]$. Since $\lambda, \gamma \in k^{*}$, this contradicts that 
$P_2(X,Z)$ is monic in $Z$. Therefore, $\mu=0$ and 
\begin{equation}\label{x2}
x_2 = \lambda x_1
\end{equation}
for some $\lambda \in k^{*}$. 

We now show that $d_1=d_2$. 
Suppose, if possible, that  $d_1 > d_2$, 
Using  (\ref{kxz}) and (\ref{x2}), we have, 
$x_1^{d_1}B \cap k[x_1,z_1] = x_2^{d_1}B \cap k[x_2,z_2]$, i.e., 
$$
(x_1^{d_1}, P_1(x_1,z_1))k[x_1,z_1] = (x_2^{d_1}, x_2^{d_1 - d_2}P_2(x_2,z_2))k[x_2,z_2].
$$
Therefore, $P_1(x_1,z_1) \in (x_2^{d_1}, x_2^{d_1 - d_2}P_2(x_2,z_2))k[x_2,z_2] \subseteq x_1k[x_1,z_1]$, 
%
which contradicts  that $P_1(X,Z)$ is a monic polynomial in $Z$. Hence, $d_1 \le d_2$ and by symmetry, we have 
$$
d_1= d_2= d \text{~say}.
$$
 Thus, we have, 
$$
x_1^{d}B \cap k[x_1,z_1] = x_2^{d}B \cap k[x_2,z_2],
 $$
i.e., 
\begin{equation}\label{pi}
 (x_1^{d}, P_1(x_1,z_1))k[x_1,z_1]  = (x_2^{d},  P_2(x_2,z_2))k[x_2,z_2].
\end{equation}
Thus $P_2(x_2,z_2) = \tau' P_1(x_1,z_1) + x_1^{d} f'$ for some $\tau', f' \in k[x_1, z_1]$. 
Since $P_i(X,Z)$'s are monic in $Z$ (for $i=1,2$),  using (\ref{z2}) and  (\ref{x2}),
 we see that 
$$
r_1 (=\deg_Z P_1)= r_2 (=\deg_Z P_2) =r \text{~say},
$$  and $\tau'  \equiv \gamma^r \mod x_1^{d}k[x_1,z_1]$. 
  Let $\tau= \gamma^r (\in k^{*}).$
 Replacing $\tau'$ by $\tau$, we have, 
 \begin{equation}\label{P2}
  P_2(x_2,z_2) = \tau P_1(x_1,z_1) + x_1^{d} f(x_1,z_1)
 \end{equation}
 for some $f \in k[x_1,z_1]$.
In particular, using (\ref{z2}) and (\ref{x2}),  and putting $x_1=0$ in  (\ref{P2}), we have, 
$$
P_2(0,\gamma z_1 + \delta(0)) = \tau P_1(0,z_1).
$$
Now we have,
\begin{equation}\label{y2}
y_2=\frac{P_2(x_2,z_2)}{x_2^{d}}= \frac{ \tau P_1(x_1,z_1) + x_1^{d} f(x_1,z_1)}{(\lambda x_1)^{d}} = \nu y_1 + g(x_1,z_1),  
\end{equation} 
where  $\nu= {\lambda ^{-d}}\tau \in k^*$ and  $g= {\lambda ^{-d}}f \in  k[x_1,z_1]$. Therefore, 
\begin{equation}\label{kxyz}
k[x_1,y_1,z_1]= k[x_2,y_2,z_2]. 
\end{equation}
We now show that $e_1=e_2$. Suppose, if possible, $e_1  > e_2$. Using (\ref{x2}) and above, we have 
$$
x_1^{e_1}B\cap  k[x_1,y_1,z_1]= x_2^{e_1}B \cap k[x_2,y_2,z_2],
$$
i.e., 
$$
(x_1^{e_1}, Q_1(x_1,y_1,z_1))k[x_1,y_1,z_1] = (x_2^{e_1}, x_2^{e_1 - e_2}Q_2(x_2,y_2,z_2))k[x_2,y_2,z_2].$$ 
Therefore, $Q_1(x_1,y_1,z_1) \in x_2k[x_2,y_2,z_2]= x_1 k[x_1,y_1,z_1]$, which contradicts that 
$Q_1(X,Y,Z)$ is monic in $Y$. Hence, $e_1 \le e_2$ and by symmetry, we have 
$$
e_1=e_2=e \text{~say}.
$$
 Thus, we have, 
$$
x_1^{e}B \cap k[x_1,y_1, z_1] = x_2^{e}B \cap k[x_2,y_2, z_2],
$$
 i.e.,
\begin{equation}\label{qi}
(x_1^{e},Q_1(x_1,y_1,z_1))k[x_1,y_1,z_1] = (x_2^{e},Q_2(x_2,y_2,z_2))k[x_2,y_2,z_2]. 
\end{equation}
Thus, $Q_2(x_2,y_2,z_2) = \kappa' Q_1(x_1,y_1,z_1) + x_1^{e} h'$ for some $\kappa', h' \in k[x_1, y_1, z_1]$. 
  Since $Q_i(X,Y,Z)$'s are monic  in $Y$, using (\ref{z2}) and (\ref{y2}),
 we see that 
$$
s_1 (=\deg_Y Q_1)= s_2 (=\deg_YQ_2) =s \text{~say},
$$
and 
 $\kappa'  \equiv \nu^s \mod x_1^{e}k[x_1,y_1,z_1]$.  Let $\kappa= \nu^s \in k^*$. 
 Replacing $\kappa'$ by $\kappa$, we have,
 \begin{equation}\label{Q2}
  Q_2(x_2,y_2, z_2) = \kappa Q_1(x_1,y_1, z_1) + x_1^{e} h(x_1,y_1,z_1)
 \end{equation}
 for some $h \in k[x_1,y_1, z_1]$.
In particular, using (\ref{z2}),(\ref{x2}),(\ref{y2}) and putting $x_1=0$ in  (\ref{Q2}), we have, 
$$ 
Q_2(0,\nu y_1 + g(0,z_1), \gamma z_1 + \delta(0)) = \kappa Q_1(0,y_1,z_1).
$$
Hence,
\begin{equation}\label{t2}
t_2=\frac{Q_2(x_2,y_2,z_2)}{x_2^{e}}= \frac{ \kappa Q_1(x_1,y_1, z_1) + x_1^{e} h}{(\lambda x_1)^{e}} = 
\frac{(\kappa t_1 + h)}{\lambda ^{e}}.
\end{equation} 

\smallskip 

Conversely, suppose conditions (I) and (II) hold.
Consider the $k$-algebra map\\ $\phi: k[X,Y,Z,T] \to B_1$ defined by
\begin{eqnarray*}\label{cisom}
\phi(X)&=&\lambda x_1,\\
\phi(Z)&=& \gamma z_1+ \delta(x_1),\\
\phi(Y)&=& \nu y_1+ g(x_1, z_1),\\
\phi(T)&=& \theta t_1 +h''(x_1, y_1, z_1),
\end{eqnarray*}
 where $\nu= \lambda^{-d}\gamma^r$, $\theta= {\lambda ^{-e}}\nu^s$, $g(x_1, z_1)= \lambda^{-d}f(x_1, z_1)$ and 
$h''(x_1, y_1, z_1)= {\lambda ^{-e}}h(x_1, y_1, z_1)$. Then clearly, 
$$
\phi(X^{d}Y - P_2(X,Z))=\phi(X^{e}T - Q_2(X,Y,Z))=0.
$$ 
Thus $\phi$ induces a $k$-linear map $\bar{\phi}: B_2\to B_1$, which is surjective. 
Since both $B_1$ and $B_2$ are of the same dimension, we have 
$\bar{\phi}$ is an isomorphism.
\end{proof}

It follows from the above result that no member of the family of double Danielewski
surfaces is isomorphic to a member of the family of Danielewski surfaces. 

\begin{cor}\label{cdistinct}
Let $A$ be any Danielewski surface, i.e.,   $
A= \dfrac{k[X,Z, V]}{(X^nV- f(X, Z))},
$
where $n \geq 2$, $\deg_Zf(X,Z) \geq 2$ and $f(0,Z) \neq 0$.
Let $B$ be a double Danielewski surface, i.e., $B$ be as in (\ref{B}) with the parameters $r \ge 2$ and $s\ge 2$. 
Then $A$ is not isomorphic to $B$. 
\end{cor}
\begin{proof} 
%
%
Let $f(X,Z)= f_0(Z) + Xf_1(X,Z)$. 
Then the ring $A\cong A'$, where
$$
A'= k[X,Z,Y,V]/(XY-f_0(Z), X^{n-1}V-Y-f_1(X,Z)).
$$
Note that the ring $A'$ is of the form in (\ref{B}) with $\deg_Y (Y-f_1(X,Z))=1$.
Hence, by Theorem \ref{isomclass}, $B \ncong A'$, as the $\deg_Y Q(X,Y,Z)=s\ge 2$. 
\end{proof}

Below we deduce a few properties of automorphisms of double Danielewski surfaces.
\begin{thm}\label{auto}
Let $B$ be as in (\ref{B}) and let the parameters $r, s$ and $e$ satisfy the
conditions (\ref{mlc}) of Lemma \ref{mll}.  Let $R$ denote the subring $k[x,y,z]$ of $B$. Let $\psi \in \aut(B)$. Then:
\begin{enumerate}
	\item[\rm(i)] $\psi(k[x,z]) = k[x,z]$.
	\item[\rm(ii)] $\psi(x) =\lambda x$ for some $\lambda \in k^*$.
	\item[\rm(iii)] $\psi((x^{d},P(x,z))k[x,z]) = (x^{d},P(x,z))k[x,z]$.
	\item[\rm(iv)] $\psi(k[x,y,z])= k[x,y,z]$. 
	\item[\rm(v)] $\psi((x^{e},Q(x,y,z))R) = (x^{e},Q(x,y,z))R$.
	\item[\rm(vi)] $\psi(t)=at+b$, where $a \in k^*$ and $b \in R$. 
         \end{enumerate}
\end{thm}
\begin{proof}
 Follows from the proof of Theorem \ref{isomclass} (see (\ref{kxz}), (\ref{x2}), (\ref{pi}), (\ref{kxyz}), (\ref{qi}), (\ref{t2})).
\end{proof}

The next result gives a characterization of any automorphism of $B$.
\begin{thm}\label{autoauto}
Let $B$ be as in (\ref{B}) and let the parameters $r, s$ and $e$ satisfy the
conditions (\ref{mlc}) of Lemma \ref{mll}.
Let $\psi$ be an endomorphism  of $B$ satisfying (i) and (ii) of Theorem \ref{auto}. 
Then $\psi$ is an automorphism of the ring $B$.
\end{thm}
\begin{proof}
As $B$ is a Noetherian ring, it is enough to show that, $\psi(B) = B$.
Since $B$ is generated by $x,y,z$ and $t$, by (i), it is enough to show that 
$y, t \in \psi(B)$. 

Since $\psi(x)= \lambda x$ for some $\lambda \in k^{*}$ and $\psi(k[x,z])= k[x,z]$, we have 
$\psi(z)=\lambda_2 z +\mu_2$ for some $\lambda_2 \in k^{*}$ and $\mu_2\in k[x]$.
Since $\psi$ is an endomorphism, we have $\psi (x^dy)=\psi(P(x,z))$.
Therefore,
  \begin{align}\label{auty}
\lambda^d x^d \psi(y)&= P(\psi(x), \psi(z))\nonumber\\
&= P(\lambda x, \lambda_2 z +\mu_2)\nonumber \\
&= \tau P(x,z)+ f(x,z)\nonumber \\
&=\tau x^dy +f(x,z), 
 \end{align}
where $\tau= \lambda_2^r \in k^*$, $f(x,z) \in k[x,z]$ and $\deg_Z f(X,Z)< \deg_ZP(X,Z)=r$.  
From (\ref{auty}), we have $f(x,z) \in x^dB \cap k[x,z]=(x^{d},P(x,z))k[x,z]$. 
Since $\deg_Z f(X,Z)< \deg_ZP(X,Z)$ and $P(X,Z)$ is  monic in $Z$, it follows that $f(x,z) = x^d g(x,z)$
for some $g(x,z) \in k[x,z]$. Hence, from (\ref{auty}), we have 
$$
y=\tau^{-1}(\lambda^d \psi(y)-g(x,z))\in k[x,z, \psi(y)] \subseteq \psi(B).
$$
Thus $k[x,y,z]=k[x,\psi(y), z]$.
Now $\psi(x^et)=\psi(Q(x,y,z))$, and hence
\begin{align}\label{autt}
\lambda^e x^e\psi(t)&= Q(\psi(x),\psi(y), \psi(z))\nonumber\\
&= Q(\lambda x, \lambda^{-d}(\tau y+g(x,z)), \lambda_2 z +\mu_2)\nonumber\\
&= \nu Q(x,y,z) + h'(x,y,z)\nonumber\\
&= \nu x^e t +h'(x,y,z), 
\end{align}
where $\nu= (\lambda^{-d}\tau)^s$, $h'(x,y,z) \in k[x,y,z]$ and 
$\deg_Y h'(X, Y,Z) <\deg_Y Q(X,Y,Z)=s$. From (\ref{autt}),
we have $h'(x,y,z) \in x^e B \cap k[x,y,z]=(x^e, Q(x,y,z))k[x,y,z]$.
Since $\deg_Yh'(X,Y,Z)<\deg_YQ(X,Y,Z)$ and $Q(X,Y,Z)$ is monic in $Y$,
we have $h'(x,y,z)= x^e h(x,y,z)$ for some $h(x,y,z) \in k[x,y,z]$.
Hence, from (\ref{autt}), we have 
$$
t=  \nu^{-1}(\lambda^e\psi(t)-h(x,y,z)) \in k[x,y,z, \psi(t)] \subseteq \psi(B).
$$
\end{proof}

We now prove a stable isomorphism property of double Danielewski surfaces.

\begin{thm}\label{stable}
Let 
$$
B_{d, e}= \dfrac{k[X,Y,Z,T]}{(X^{d}Y - P(X,Z), X^{e}T - Q(X,Y,Z))}, 
$$
where $d, e \in \bN$, $P(X,Z)$ is a monic polynomial in $Z$ with $\deg_Z P(X,Z)=r\ge 2$ and $Q(X,Y,Z)$ 
is a monic polynomial in $Y$ with $\deg_YQ(X,Y,Z)=s\ge 2$. 
Let 
$$
P(X,Z)= a_0(X)+a_1(X)Z+ \dots+ a_{r-1}(X)Z^{r-1}+Z^r, 
$$
and 
$$
Q(X,Y,Z)=b_0(X,Z) + b_1(X,Z)Y+ \dots+ b_{s-1}(X,Z)Y^{s-1} + Y^s.
$$
Define
$$
P'(X, Z): = a_1(X)+ \dots+ (r-1)a_{r-1}(X)Z^{r-2}+rZ^{r-1}
$$ 
and
$$
Q'(X,Y,Z)=b_1(X,Z) + 2b_2(X,Z)Y+ \dots+ (s-1)b_{s-1}(X,Z)Y^{s-2} + sY^{s-1}.
$$
Suppose that 
$$
(P(0,Z), P^{'}(0,Z))k[Z]=k[Z] \text{~and~} (P(0,Z), Q(0,Y,Z), Q'(0,Y,Z))k[Y,Z]= k[Y,Z].
$$ 
Then, for  $e\ge 2$, 
$$
{B_{d, e}}^{[1]} \cong {B_{d, e-1}}^{[1]}. 
$$
\end{thm}
\begin{proof}
We write $B_{d,e}=B$ for notational convenience. As before, let $x$, $y$, $z$ and $t$ respectively denote 
the images of $X$, $Y$, $Z$ and $T$ in $B$.
Let $\phi: B \to B[U]$ be an exponential map defined on $B$ by  
\begin{eqnarray*}\label{}
\phi(x)&=&x,\\
\phi(z)&=& z+ x^{d+e}U,\\
\phi(y)&=& \dfrac{P(x,z+x^{d+e}U)}{x^{d}}= y + x^{e}U\alpha(x,z,U),\\
\phi(t)&=&\dfrac{Q(x,y+x^{e}U\alpha(x,z,U),z+x^{d+e}U)}{x^{e}}= t+ U\beta(x,y,z,U),
\end{eqnarray*} 
where $\alpha(x,z,U) \in k[x,z,U], \beta(x,y,z,U) \in k[x,y,z,U].$	
Let $A=B[w] = B^{[1]}$ and extend $\phi$ to $A$ by defining $\phi(w)= w - xU.$ 
Let $f = x^{d+e-1}w + z$. Then $f \in A^{\phi}.$ 
Now, 
$$
P(x,f)- P(x,z) = x^{d+e-1}(P'(0,z)w+ x \theta)
$$
for some $\theta \in A.$ Therefore, 
\begin{align}\label{pxf}
	P(x,f) &=P(x,z)+x^{d+e-1}(P'(0,z)w+ x \theta) \nonumber\\
	       &= x^{d}y + x^{d+e-1}(P'(0,z)w+ x\theta)  \nonumber\\
	       &=x^{d}g,
\end{align} 
where $g= y+ x^{e-1}(P'(0,z)w+ x \theta) \in A.$
Note that, since $x^{d}g=P(x,f) \in A^{\phi}$ and 
$A^{\phi}$ is factorially closed in $A$ (cf. Lemma \ref{exp3}), we have $g\in A^{\phi}.$ 
	Now 
	\begin{align*}
	Q(x,g,f)&=Q(x, y+ x^{e-1}(P'(0,z)w+ x \theta),z+  x^{d+e-1}w)\\
	            &=Q(x,y,z)+ x^{e-1}P'(0,z)Q'(0, y, z)w+x^{e}\rho
	\end{align*}
	for some $\rho \in A.$ Therefore,
	\begin{align}\label{qxgf}
	Q(x,g,f) &= x^{e}t + (Q(x,g,f)- Q(x,y,z))  \nonumber\\
	         &= x^{e}t + x^{e-1}(P'(0,z)Q'(0, y, z)w+x\rho) \nonumber\\
	         &= x^{e-1}(P'(0,z)Q'(0, y, z)w+xt + x\rho)  \nonumber\\
	         &= x^{e-1}h
	\end{align} 
where $h= P^{'}(0,z)Q^{'}(0,y,z)w +xt+ x\rho\in A.$ 
Note that, since $Q(x,g,f)= x^{e-1}h \in A^{\phi}$ 
and $A^{\phi}$ is factorially closed in $A$, we have 
$h \in A^{\phi}.$ 
From the given condition we have, $(Q(0,Y,Z),Q^{'}(0,Y,Z)P^{'}(0,Z), P(0,Z))k[Y,Z]=k[Y,Z].$
 Then there exist $a(Y,Z),b(Y,Z),c(Y,Z) \in k[Y,Z]$ such that 
$$ 
Q^{'}(0,Y,Z)P^{'}(0,Z)a(Y,Z) + Q(0,Y,Z)b(Y,Z) + P(0,Z)c(Y,Z) =1.
$$ 
Since $Q(0, y, z) \in xA$ and $P(0, z) \in xA$, we have,
\begin{equation}\label{starr}
Q'(0,y,z)P'(0,z)a(y,z)+ x \delta=1
\end{equation}
for some $\delta \in A$. 

Let $v= \dfrac{w - a(g,f)h}{x}$. We first show that $v \in A.$ 
Note that $a(g, f)-a(y,z)\in xA$. Let $a(g, f)-a(y,z)=x\gamma$ for some $\gamma \in A$. 
 Now
	\begin{align*}
	w-h a(g,f)&=  w-h a(y,z) -h(a(g, f)-a(y,z))\\
	               &= w-ha(y,z)-hx\gamma\\
	               &= w - a(y,z)(P^{'}(0,z)Q^{'}(0,y,z)w + xt+x\rho) -hx\gamma\\
	              &= w(1-a(y,z)P^{'}(0,z)Q^{'}(0,y,z)) - x(a(y,z)t+ a(y,z)\rho+h\gamma)\\
	              &= wx\delta-x(a(y,z)t+ a(y,z)\rho+h\gamma)\in xA
	\end{align*} 
Thus,  $v = \dfrac{w-h a(g,f)}{x} \in A.$ 
Now, since $x,f,g,h \in A^{\phi}$, we have  $\phi(v) = v -U.$ Then, by Lemma \ref{exp3}(iii), 
\begin{equation}\label{slice}
A= A^{\phi}[v]= {(A^{\phi})}^{[1]}. 
\end{equation}
Let $E= k[x,f,g,h].$ Consider indeterminates $X$, $F$, $G$ and $H$  over $k$ so that 
$k[X,F,G,H]= k^{[4]}$ and let  
\begin{equation}\label{ismo}
E_1= \dfrac{k[X,F,G,H]}{(X^{d}G - P(X,F), X^{e-1}H - Q(X,G,F))} \cong B_{d,e-1}.
\end{equation}
We first show that $E \cong E_1$. 
Clearly there exists a surjective $k$-algebra homomorphism 
$\Phi: k[X,F,G,H]\to E$
such that $\Phi(X)=x$, $\Phi(F)=f$, $\Phi(G)=g$ and $\Phi(H)=h$.
Using (\ref{pxf}) and (\ref{qxgf}), we see that $\Phi$ 
induces a surjective $k$-algebra map $\bar{\Phi}: E_1 \to E$.
Note that $A[1/x]=E[1/x][w]=E[1/x]^{[1]}$ and hence dimension of $E$ is two. 
Therefore, as $E_1$ is an integral domain 
(cf. Lemma \ref{ID}) of dimension $2$, $\bar{\Phi}$ is an isomorphism.

We now show that $A^{\phi}= E$.  
Clearly $E \subseteq A^{\phi}$. Since $A[1/x]=E[1/x][w]$, it follows that  
$E[1/{x}]= A^{\phi}[1/x].$ 
Therefore, to show that $E= A^{\phi},$ it is enough to show that $xE=xA^{\phi} \cap E$.
Since $xA^{\phi} = xA \cap A^{\phi}$ by Lemma \ref{exp3}(i), it is enough to show that
$xE= xA \cap E$, i.e., to show that the kernel of the map 
$\iota : E \rightarrow A/xA$
is $xE$. For $u \in A$, let $\tilde{u}$ denote the image of $u$ in $A/xA$. 
Note that 
$$
A/xA=\dfrac{k[Y,Z,T,W]}{(P(0, Z),Q(0,Y,Z))}=\left(\dfrac{k[Y,Z]}{(P(0, Z),Q(0,Y,Z))}\right)[W, T]=k[\tilde{z}, \tilde{y}, \tilde{w}, \tilde{t}].
$$
Clearly $\iota(f)= \tilde{z}$, $\iota(g)=\tilde{y}$ and $\iota(h)=P'(0,\tilde{z})Q'(0, \tilde{y}, \tilde{z}) \tilde{w}$.
Now by (\ref{starr}), $P'(0,\tilde{z})Q'(0, \tilde{y}, \tilde{z}) \in (A/xA)^{*}$.
Hence, $\iota(E)= k[\tilde{z}, \tilde{y}, \tilde{w}]$ and $A/xA= \iota(E)[\tilde{t}]= \iota(E)^{[1]}$.
Thus $\dim \iota(E)=\dim (A/xA)-1=1$. 
Now as 
$$
E/xE \cong E_1/xE_1 = k[F,G,H]/(P(0,F),Q(0,G,F))= \left(\dfrac{k[F,G]}{(P(0,F),Q(0,G,F))}\right)^{[1]}, 
$$
we have 
$\dim (E/xE)= 1= \dim \iota(E)$. Hence kernel of $\iota$ is $xE$.  
Thus $A^{\phi}=E$. Hence by  (\ref{slice}), we have
$
B_{d,e}[w] =A= {(A^{\phi})}^{[1]}=E^{[1]}\cong E_1^{[1]},
$
i.e., ${B_{d,e}}^{[1]}\cong{B_{d,e-1}}^{[1]}$.
\end{proof}

By Theorems \ref{isomclass} and \ref{stable}, it follows:
\begin{cor}\label{cpd}
Under the hypotheses of Theorem \ref{stable}, for every $d, e \in \bN$, $$B_{d, e}\ncong B_{d, e+1} {\text~but}~~{B_{d, e}}^{[1]}\cong {B_{d, e+1}}^{[1]}.$$ 
Thus the rings $B_{d,e}$ provide counter-examples to the Cancellation Problem.
\end{cor}

\medskip
{\bf Acknowledgements.} The authors thank Prof. Amartya Kumar Dutta for suggesting 
several changes to improve the exposition and simplifying the proofs. 
The first author acknowledges Department of Science and Technology for their SwarnaJayanti Fellowship.
The second author acknowledges Council of Scientific and Industrial Research (CSIR) for their research grant.

\end{document}